\documentclass[amsmath,amssymb,amsthm]{revtex4-1}

\usepackage{graphicx,amsthm}

\newcommand{\reff}[1]{{\rm (\ref{#1})}}

\newcommand{\R}{\mathbb{R}} 
 
\newcommand{\bX}{\mathbf X} 
\newcommand{\bx}{\mathbf x}

\newcommand{\bn}{\mathbf n} 
 
\newcommand{\ve}{\varepsilon}

\newtheorem{theorem}{Theorem}[section]

\numberwithin{equation}{section}
\numberwithin{table}{section}

\begin{document}
\preprint{APS/123-QED}

\title{A New Phase-Field Approach to Variational Implicit Solvation of Charged Molecules with the Coulomb-Field Approximation}

\author{Yanxiang Zhao}
\email{yxzhao@email.gwu.edu}
\thanks{}
\affiliation{ 
Department of Mathematics, the George Washington University, 801 22nd St. NW, Phillips 739, Washington, DC, 20052}

\author{Yanping Ma}
\email{yma@lmu.edu}
\thanks{}
\affiliation{ 
Department of Mathematics, Loyola Marymount University, 1 LMU drive, Los Angeles, CA 90045, USA}

\author{Hui Sun}
\email{hui.sun@csulb.edu}
\thanks{}
\affiliation{ 
Department of Mathematics and Statistics, California State University, Long Beach, CA 90840-1001, USA}

\author{Bo Li}
\email{bli@math.ucsd.edu}
\thanks{}
\affiliation{Department of Mathematics and Quantitative Biology Graduate Program, 
University of California, San Diego, 9500 Gilman Drive, Mail code: 0112, La Jolla, CA 92093-0112, USA}

\author{Qiang Du}
\email{qd2125@columbia.edu}
\affiliation{Department of Applied Physics and Applied Mathematics, Columbia University, NY 10027, USA}


\begin{abstract}
{\bf Abstract.} 
We construct a new phase-field model for the solvation of charged molecules with a variational implicit solvent. Our phase-field free-energy functional includes the surface energy, solute-solvent van der Waals dispersion energy, and electrostatic interaction energy that is described by the Coulomb-field approximation, all coupled together self-consistently through a phase field. By introducing a new phase-field term in the description of the solute-solvent van der Waals and electrostatic interactions, we can 
keep the phase-field values closer to those describing the solute and solvent regions, respectively, 
making it more accurate in the free-energy estimate.  We first prove that our phase-field functionals $\Gamma$-converge
to the corresponding sharp-interface limit.  We then develop and implement an efficient and stable 
numerical method to solve the resulting gradient-flow equation to obtain equilibrium 
conformations and their associated free energies of the underlying charged molecular system. 
Our numerical method combines a linear splitting scheme, spectral discretization, and exponential 
time differencing Runge-Kutta approximations.  Applications to the solvation of single ions and a two-plate system
demonstrate that our new phase-field implementation improves the previous ones by achieving the localization of the system forces near the solute-solvent interface and maintaining more robustly the desirable hyperbolic tangent profile for even larger interfacial width. This work provides a scheme to resolve the possible unphysical feature of negative values in 
the phase-field function found in the previous phase-field modeling (cf.\ H.\ Sun, {\it et al.} J. Chem. Phys., 2015)  
of charged molecules with the Poisson--Boltzmann equation for the electrostatic interaction. 
\end{abstract}

\maketitle


\section{Introduction}
\label{s:introduction}

We consider the solvation of charged molecules in an aqueous solvent (i.e., water or salted water). 
The entire region $\Omega$ of an  underlying solvation system consists 
of a solute (i.e., the charged molecule) region
$\Omega_{\rm m}$ ($\rm m$ stands for charged molecules), a solvent region 
$\Omega_{\rm w}$ ($\rm w$ stands for water), and a solute-solvent interface $\Gamma$ that 
separates these two regions.   cf.\ Figure~\ref{fig:SchematicSystem}. 
We assume there are $N$ solute atoms located at $\bx_1, \dots, \bx_N$ inside 
the solute region $\Omega_{\rm m}$, carrying partial charges $Q_1, \dots, Q_N$, respectively. 
This solute-solvent interface is also treated as a dielectric boundary, as the dielectric coefficient
$\ve_{\rm m}$ in the solute region is close to $1$ and that $\ve_{\rm w}$ in
the solvent region is close to $80.$

\begin{figure}[htbp]
\begin{center}
\includegraphics[width=100mm]{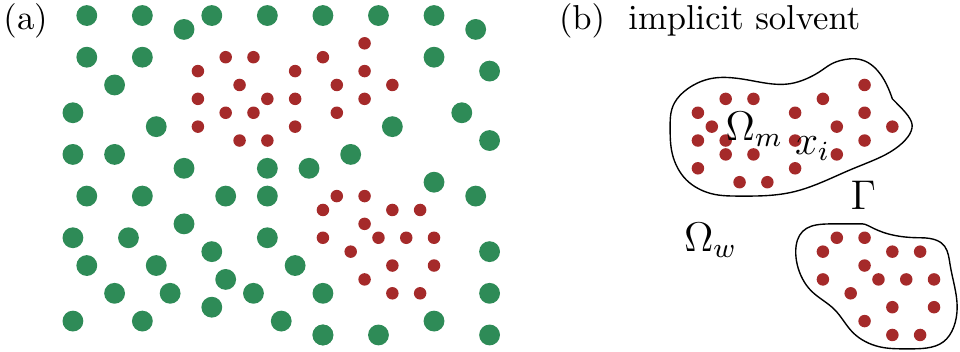}
\end{center}
\vspace{-6 mm}
\caption{\small {Schematic description of a solvation system.
(a) In a fully atomistic model, both the solute atoms (small and brown dots) 
and solvent molecules (large and green dots) are degrees 
of freedom of the system.  (b) In an implicit-solvent model, the
solvent molecules are coarse-grained and the solvent is treated as a continuum. 
The solvent region $\Omega_{\rm w}$ and the solute region $\Omega_{\rm m}$ are separated by
the solute-solvent interface (i.e., the dielectric boundary) $\Gamma.$
The solute atoms are located at $\bx_1, \dots, \bx_N$ inside $\Omega_{\rm m}$.}}
\label{fig:SchematicSystem}
\end{figure}

In a variational implicit-solvent model (VISM) \cite{DSM06a,DSM06b}
(cf.\ also \cite{WangEtal_VISMCFA_JCTC12, Zhou_VISMPB_JCTC14,SunWenZhao_JCP2015}), 
one obtains an equilibrium solute-solvent interface and a free-energy estimate by minimizing
a macroscopic solvation free-energy functional among all solute-solvent interfaces $\Gamma$. 
Such a functional includes the solute-solvent interfacial energy, 
solute-solvent van der Waals interaction energy, and the electrostatic free energy, all
determined by the interface $\Gamma.$  The electrostatic part of the free energy is often described by 
the Poisson--Boltzmann (PB) theory 
\cite{DavisMcCammon90,CDLM_JPCB08,LiChengZhang_SIAP11,Li_SIMA09,SharpHonig90,Zhou_VISMPB_JCTC14}
or the Coulumb-field approximation (CFA) \cite{BashfordCase00,WangEtal_VISMCFA_JCTC12}. 

In this work, we consider the phase-field implementation of VISM 
\cite{LiZhao_SIAP13,SunWenZhao_JCP2015,Zhao_2013,DaiLiLu_ARMA2017}. 
We use a phase field $\phi: \Omega \to \R$ to describe the solute-solvent interface
with $ \{ \phi \approx 1 \}$ and $\{ \phi \approx 0 \}$ representing the solute and solvent regions, respectively. 
The corresponding solvation free-energy functional of a phase field  $\phi: \Omega \to \R$ is given by 
\begin{align}
\label{General_phasefield}
F^{\epsilon}[\phi] =  \gamma  \int_{\Omega} 
\left[\frac{\epsilon}{2}|\nabla \phi |^2 + \frac{1}{\epsilon}W(\phi )\right]\, d\bx 
+ \rho_{\rm w} \int_{\Omega} f(\phi)  U_{\rm vdW}  \, d \bx
+ \int_{\Omega} f(\phi) U_{\rm ele}(\bx) \, d \bx. 
\end{align}
Here, $\epsilon > 0$ is a small parameter that controls the width of solute-solvent
interfacial region. The first term describes the solute-solvent interfacial energy, 
where $\gamma > 0$ is the surface tension (a given constant) and 
\[
W(\phi) = 18(\phi^2-\phi)^2. 
\]
The specific constant $18$ is chosen for convenience of analysis; cf.\ Section~\ref{s:GammaConvergence}. 

The second term  describes the solute-solvent van der Waals interaction. 
In this term, $\rho_{\rm w}$ is the bulk solvent density (a given constant) and 
\begin{align}\label{UvdW}
U_{\text{vdW}} (\bx) = \sum_{i=1}^N U_{\rm LJ}^ {(i)}(|\bx - \bx_i|),  
\end{align}
where each $U_{\rm LJ}^{(i)}$ is taken to be a Lennard-Jones potential
\begin{align*}
U^{(i)}_{\text{LJ}}(r)=4\varepsilon_i
\left[ \left(\dfrac{\sigma_i}{r}\right)^{12}-\left(\dfrac{\sigma_i}{r}\right)^6\right],
\end{align*}
with $\varepsilon_i$ and $\sigma_i$ being the corresponding interaction
energy and linear size of atomic excluded volume. The function $f (\phi)$ has the property that 
\begin{equation}\label{f01} 
f(0) = 1 \qquad \mbox{and} \qquad  f(1) = 0,
\end{equation}
indicating that the integral is taken over the solvent region. 

The last term is the electrostatic energy, where  $U_{\rm ele}$ is the electrostatic energy density and 
the integral is again taken over the solvent region.  For the PB electrostatics, one needs to 
solve a phase-field dielectric boundary PB equation to obtain the electrostatic energy density $U_{\rm ele}$ 
\cite{LiLiu_SIAP15, SunWenZhao_JCP2015, DaiLiLu_ARMA2017}. 
Here, we shall consider the CFA, which yields a good approximation of the electrostatic free energy
when the ionic effect is less significant.  The CFA  makes the computation efficient, and also provides 
a simple model for analyzing the geometry of interfacial region. 
In the CFA \cite{ChengChengLi_Nonlinearity11,WangEtal_VISMCFA_JCTC12, Zhao_2013}, 
the electrostatic energy density is given by
\begin{equation}\label{Uele}
U_{\rm ele} (\bx) = \frac{1}{32 \pi^2 \ve_0}
\left( \frac{1}{ \ve_{\rm w}} - \frac{1}{\ve_{\rm m}}\right) 
 \left| \sum_{i=1}^N \frac{Q_i (\bx-\bx_i) }{|\bx - \bx_i|^3} \right|^2, 
\end{equation}
where $\varepsilon_{\text{0}}$ is the vacuum permittivity. 

The minimization of the free-energy functional \reff{General_phasefield} can 
be achieved by solving for a steady-state solution of 
the corresponding gradient-flow equation 

\begin{align}
\label{General_gradientflow}
 \partial_t \phi = \gamma \left[ \epsilon \Delta \phi - \frac{1}{\epsilon} W'(\phi) \right] 
- f'(\phi) \left( \rho_{\text{w}}  U_{\text{vdW}} + U_{\text{ele}} \right),  
\end{align}
with a fixed and small $\epsilon > 0$, and some initial and boundary conditions for $\phi$.

The form of the function $f(\phi)$ is crucial to capturing  the interfacial structure of an underlying interface system. 
An ad hoc choice of such a function, such as 
\begin{equation}\label{f03}
f(\phi) = (\phi-1)^2,
\end{equation}
may lead to some unphysical features, such as the non-monotonicity of the phase-field 
functions from 0 to 1 and the loss of  localization of the force 
near the interface \cite{Zhao_2013, SunWenZhao_JCP2015}.  In this work, we propose a new form of this function 
\begin{align}
\label{f_term}
f(\phi) = (\phi^2-1)^2. 
\end{align}
We will demonstrate numerically that, with such a function $f$,  the energy-minimizing 
phase-field approximates better $1$ and $0$, in the two regions, respectively.  
Heuristically, with such a function, we have not only  \reff{f01}, but also that 
\begin{equation}\label{f02}
f'(0) = 0 \qquad \mbox{and} \qquad f'(1) = 0. 
\end{equation}
These will lead to a more localized ``boundary force" near the solute-solvent interface that 
involves $f'(\phi)$,  which is consistent with the force balance equation (Euler--Lagrange equation)
for a sharp interface  \cite{LiZhao_SIAP13}.  Moreover, the localization of force due to the property 
\reff{f02} allows us to use a small computational box that encloses the entire solute region and solute-solvent 
interface, thus greatly improving the computational efficiency. 
Notice that the issue of non-monotonic artificial interfacial structure does not exist, 
if one only minimizes the surface energy, i.e., the first integral in \reff{General_phasefield}. This issue arises 
from the nonlocality of the van der Waals energy and the electrostatic energy, the last two integrals in \reff{General_phasefield}. 

We shall first prove the $\Gamma$-convergence of our new, phase-field free-energy
functionals to the corresponding sharp-interface limit as $\epsilon \to 0$. 
This is similar to the proof given in \cite{LiZhao_SIAP13}, cf.\ also  \cite{DaiLiLu_ARMA2017}. 
We then design, implement, and test accurate and efficient numerical methods for 
solving the gradient-flow  equation.  Our methods couple a linear splitting scheme 
\cite{DuZhu_2005, Ju_JSC2015, Xu_SIAM2006, Yang_JCP2007, WangJuDu_JCP16}, 
spectral discretization schemes,  and exponential time differencing Runge-Kutta approximations 
\cite{Cox_JCP2002, Kassam_SIAMSC2005, JuZhangDu_CMS15, WangJuDu_JCP16}.   
We finally apply our model and numerical methods to  some charged molecules, such a single ion and 
a two-plate system, demonstrating that our proposed new model performs numerically better than the 
pervious ones by achieving the force localization near the solute-solvent interface and 
maintaining more robustly the desirable hyperbolic tangent profile for even larger interfacial width.

The variational implicit-solvent model (VISM), implemented with a robust level-set method, has successfully predicted dry and wet states and dewetting transition, charge effects, and potential of mean forces, and many other important properties of biological molecules that have been observed in experiment and in molecular dynamics simulations \cite{CDML_JCP07, ChengLiWang_JCP10, WangEtal_VISMCFA_JCTC12, Zhou_VISMPB_JCTC14, Zhou_StochLSM2016,Ricci_MVISM_2017, Che_Pocket_2015, Che_p53_JCTC2014}. The phase-field implementation of VISM provides an alternative mathematical model for the computation of molecular conformations and free energies. Moreover, it may be used to include bulk solvent fluctuations that together with the solute-solvent interface fluctuations enable an underlying system to make transition from one equilibrium conformation to another \cite{Metiu_GLModel_JCP83,KarmaRappel_PRE99}. This is particularly important in terms of hydrophobic interactions \cite{Chandler05,BerneWeeksZhou_Rev09,Luzar_PCCP2011}.

The rest of the paper is organized as follows: 
In Section~\ref{s:GammaConvergence}, we prove the $\Gamma$-convergence of our phase-field
functionals \reff{General_phasefield} to the corresponding sharp-interface limit. 
 In Section~\ref{s:NumericalMethods}, we describe our numerical methods 
for solving the gradient-flow dynamics equation of the phase-field free-energy functional. 
Finally, in Section~\ref{s:Tests}, we apply our theory and methods to the solvation of single ions and 
a two-plate system. The Appendix contains some details of our numerical methods. 


\section{$\Gamma$-Convergence}
\label{s:GammaConvergence}

In this section, we will briefly discuss the $\Gamma$-convergence of the phase-field model \reff{General_phasefield} to the corresponding sharp-interface model by following the approach similar to that  in \cite{LiZhao_SIAP13}. 
To make our results more general, we consider as in \cite{LiZhao_SIAP13} in this section the following functional of both phase field and the set of solute particles, including the solute-solute mechanical interactions: 
\begin{align}\label{General_phasefield2}
F^{\epsilon}[\bX, \phi] = E[\bX] + \gamma  \int_{\Omega} 
\left[\frac{\epsilon}{2}|\nabla \phi |^2 + \frac{1}{\epsilon}W(\phi )\right]\, d\bx 
+ \int_{\Omega} f(\phi)  U(\bX, \bx)\, d \bx, 
\end{align}
where 
\[
U(\bX,\bx)= \rho_{\rm w} U_{\rm vdW} (\bX, \bx) + U_{\rm ele}(\bX, \bx),
\]
and $E = E[\bX]$ is the potential energy of molecular mechanical interactions of solute atoms located at $\bx_1, \dots, \bx_N$ inside the solute region $\Omega_{\rm m}$ (cf.\ Figure~\ref{fig:SchematicSystem}) and $\bX = (\bx_1, \dots, \bx_N) $. The
terms $U_{\rm vdW}(\bX, \bx)$ and $U_{\rm ele}(\bX, \bx)$ are exactly the same as $U_{\rm vdW} (\bx) $ defined
in \reff{UvdW} and $U_{\rm ele}(\bx)$ defined in \reff{Uele}, respectively, except we explicitly include $\bX$ to indicate the dependence on $\bX.$
The molecular mechanical interactions include the chemical bonding, bending, and torsion;  
the short-distance repulsion and the long-distance attraction;  and the Coulombic charge-charge interaction.  The corresponding sharp-interface model is written as
\begin{align}\label{General_sharpinterface2}
F[\bX, \Gamma] = E[\bX] + \gamma \text{Area}(\Gamma) + \int_{\Omega_{\text{w}}} U(\bX,\bx)\ d\bx, 
\end{align}
where $\Gamma$ represents the solute-solvent interface in the sharp-interface setting.

Let $\Omega$ be a nonempty, open, connected, and bounded subset of $\R^3$ with a Lipschitz-continuous boundary $\partial \Omega$. Let $\overline{\Omega}$ be the closure of $\Omega$ in $\R^3.$ 
Let $N \ge 1$ be an integer and denote 
\[
O_N = \left\{ \bX =  (\bx_1,\cdots, \bx_N)\in (\R^3)^N:  \bx_i\neq \bx_j\ 
\text{if}\  i\neq j\  \text{for}\ 1\le i,j\le N   \right\}. 
\]
Clearly $O_N$ is an open subset of $(\R^3)^N$. We assume that  $E: \overline{\Omega}^N \to \R \cup \{ + \infty\}$ is finite and continuous in $\Omega^N \cap O_N$, infinite in $\overline{\Omega}^N \setminus (\Omega^N \cap O_N)$ ,and has a finite lower bound $E_{\min}$ in $\overline{\Omega}^N$. We also assume 
\begin{align*}
E[\bX]\to +\infty\quad  \text{as}\  \min_{1 \le i < j \le N} |\bx_i-\bx_j|\to 0 \quad \text{or}\  \min_{1 \le i \le N} \mbox{dist}\,(
\bx_i, \partial\Omega) \to 0. 
\end{align*}
We shall assume $U(\bX,\bx): \overline{\Omega}^N \times \overline{\Omega} \to \R \cup \{ +\infty \}$ is finite and continous in $(\Omega^N \times \overline{\Omega}) \cap O_{N+1}$, infinite in $\overline{\Omega}^{N+1}
\setminus \left( ( \Omega^{N}\times \overline{\Omega} ) \cap O_{N+1} \right)$, and has a finite lower bound $U_{\min}$ in $\overline{\Omega}^N \times \overline{\Omega}$. We finally assume
\begin{align*}
U(\bX, \bx)\rightarrow+\infty\quad \text{as}\ \min_{0 \le i < j \le N} |\bx_i-\bx_j|\rightarrow 0 \quad \text{with}\ \bx_0 = \bx.
\end{align*}

We denote
\begin{align*}
\mathcal{M}_0 = \left\{ (\bX, A): X\in \overline{\Omega}^N, A \subseteq \Omega,  
A \mbox{ is Lebesgue measurable}\, \right\}. 
\end{align*}
For any $(\bX, A) \in \mathcal{M}_0$, we define 
\begin{align}\label{F0XA}
F_0[\bX,A] = E[\bX] + \gamma P_{\Omega}(A) + \int_{\Omega\backslash A} U(\bX, \bx) \, d\bx, 
\end{align}
where $P_{\Omega}(A)$, the perimeter of a set $A\subset \mathbb{R}^3$, is standardly defined by functions of bounded variation in $BV({\Omega})$ \cite{Giusti84,Ziemer_Book89,EvansGariepy_Book92}. Since $E$ and $U$ are bounded below, $F_0(\bX, A) > - \infty$. If $A\subset \Omega $ is open and smooth, with a finite perimeter in $\Omega$, then $F_0(\bX,A) = F(\bX, \Gamma)$, where $\Gamma = \partial A$ and $F$ is defined in
\reff{General_sharpinterface2} with $\Omega_{\rm w} = \Omega \setminus A$. Therefore, $F_0: \mathcal{M}_0 \to \R \cup \{ + \infty \}$ describes the free energy of a solvation system with $A$ being the solute region. 

As shown in \cite{LiZhao_SIAP13}, we have the existence of a global minimizer of the sharp-interface free energy functional $F_0: \mathcal{M}_0 \rightarrow \mathbb{R}\cup\{+\infty\}$:
\begin{theorem} \label{theorem:existence_sharpinterface}
There exists $(X,A)\in\mathcal{M}_0$ such that
\begin{align}
F_0[ X, A] = \inf_{(Y,B)\in \mathcal{M}_0} F_{0}[Y, B]. 
\end{align}
Moreover, this minimum value is finite.
\end{theorem}
We omit the proof as it is similar to that of Theorem 2.1 in \cite{LiZhao_SIAP13}. Additionally, the minimal energy in Theorem \ref{theorem:existence_sharpinterface} can be approximated by free energies of certain ``regular" subsets, see Theorem 2.2 in \cite{LiZhao_SIAP13} for details.

We now consider the functional $F^{\epsilon}$ in \reff{General_phasefield2}. Let $\mathcal{M}=\bar{\Omega}^N\times H^1(\Omega)$, and $\epsilon_0\in(0,1]$ be sufficiently small. 
Then we have the existence of a global minimizer of the  functional $F^{\epsilon}:\mathcal{M} \rightarrow \mathbb{R}\cup\{+\infty\}$
for small $\epsilon > 0$. 
\begin{theorem}
For each $\epsilon\in(0,\epsilon_0]$, there exists $(X_{\epsilon},\phi_{\epsilon})\in\mathcal{M}$ with $X_{\epsilon}\in\Omega^N\cup O_N$ such that
\begin{align}\label{eqn:existence_phasefield}
F^{\epsilon}[X_{\epsilon}, \phi_{\epsilon}] = \inf_{(X, \phi)\in\mathcal{M}} F^{\epsilon} [X,\phi],
\end{align}
and this infimum value is finite. 
\end{theorem}
\begin{proof}
The key to proving the existence of a global minimizer is to obtain the lower and upper boundedness for 
$F^{\epsilon} [X,\phi]$ for any $\epsilon\in(0,\epsilon_0]$. The upper bound is achieved easily as we can fix some $X^*$ and construct an associated $\phi^*_{\epsilon}$ such that $F^{\epsilon} [X^*,\phi^*_{\epsilon}]$ is bounded independent of $\epsilon$ (see Theorem 3.1 in \cite{LiZhao_SIAP13} for the detailed construction of $\phi^*_{\epsilon}$). For the lower bound, we have
\begin{align*}
F^{\epsilon}[X, \phi] 
&\ge E_{\rm min} + \frac{\gamma \epsilon}{2} \|\nabla\phi\|^2_{L^2(\Omega)} + 
\frac{\gamma}{2\epsilon} \int_{\Omega} W(\phi)\,dx
+  \frac{\gamma}{2\epsilon_0} \int_{\Omega} W(\phi)\, dx + 
U_{\min}\int_{\Omega}(\phi^2-1)^2\, dx\\
& =  E_{\min} + \frac{\gamma \ve}{2} \|\nabla\phi\|^2_{L^2(\Omega)} + 
\frac{\gamma}{2 \varepsilon} \| W(\phi) \|_{L^1(\Omega)} 
+ \frac{ \gamma}{2 } \| \phi \|_{L^4(\Omega)}^4 + \int_\Omega g(\phi) \, dx, 
\end{align*}
where
\[
g(\phi ) = \frac{\gamma}{2\epsilon_0}\left[ W(\phi)- \epsilon_0\phi^4 \right] + U_{\rm min}(\phi^2-1)^2. 
\]
Note that $g: \mathbb{R}\rightarrow\mathbb{R}$ is continuous, and $U_{\min}$ is finite. Hence, if  $\epsilon_0$ is sufficiently small, then  $g(s)\rightarrow +\infty$ as $|s|\rightarrow +\infty$. Then we have
\begin{align*}
F^{\epsilon}[X, \phi] 
& \ge  C + \frac{\gamma \ve}{2} \|\nabla\phi\|^2_{L^2(\Omega)} + 
\frac{\gamma}{2 \varepsilon} \| W(\phi) \|_{L^1(\Omega)} 
+ \frac{ \gamma}{2 } \| \phi \|_{L^4(\Omega)}^4
\end{align*}
with $C = E_{\min} + |\Omega|\inf_{s\in\mathbb{R}}g(s)$.

With the lower and upper bounds, we can choose a sequence of $(X_k,\phi_k)$ which is bounded in $\bar{\Omega}^N\times H^1(\Omega)$. Using the standard compactness argument, we can find a subsequence, not relabeled, that converges to $(X_{\epsilon},\phi_{\epsilon})\in (\Omega^N\cap O_N)\times H^1(\Omega)$. Finally the Fatou's lemma will yield \reff{eqn:existence_phasefield}.
\end{proof}

With the existence of global minimizers for sharp-interface energy $F_0$ and phase-field one $F^{\epsilon}$, we have the convergence of the global minimum free energies and the global free energy minimizers:
\begin{theorem}\label{t:Diffuse2Sharp}
Let $\epsilon_k \in (0, \epsilon_0] $   $ (k = 1, 2, \dots)$ be such that $ \epsilon_k \downarrow 0$. 
For each $k \ge 1$, let $(X_{\epsilon_k}, \phi_{\epsilon_k}) \in \mathcal{M}$ be such that 
\begin{equation}\label{minvek}
F^{\epsilon_k}[X_{\epsilon_k}, \phi_{\epsilon_k}]
 = \min_{(X,\phi)\in\mathcal{M}}F^{\epsilon_k}[X, \phi]. 
\end{equation}
Then there exists a subsequence of $\{ (X_{\epsilon_k}, \phi_{\epsilon_k}) \}_{k=1}^\infty$, 
not relabeled, such that $X_{\epsilon_k} \to X_0 $ in $(\R^3)^N$ for some 
$X_0 \in \Omega^N \cap O_N$ and $ \phi_{\epsilon_k}\rightarrow \chi_{A_0} $ 
in $L^{4-\lambda}(\Omega)$ for any $\lambda \in (0,1)$ and for some measurable subset 
$A_0\subseteq \Omega$ that has a finite perimeter in $\Omega$.  Moreover, 
\begin{align}\label{theorem3Eqn1}
\lim_{k\to \infty} F^{\epsilon_k}[X_{\epsilon_k}, \phi_{\epsilon_k}] = F_0[X_0, A_0]
\end{align}
and
\begin{align}
\label{theorem3Eqn2}
F_0[X_0, A_0]  = \min_{(X, A) \in \mathcal{M}_0} F_0[X, A]. 
\end{align}
\end{theorem}
The proof is omitted as it is similar to the one in  \cite{LiZhao_SIAP13}.


\section{Numerical Methods}\label{s:NumericalMethods}

\subsection{Equivalent reformulation with a linear splitting}

We first adopt an analogous linear splitting scheme that has been used  in designing stabilized numerical methods for the classical Allen--Cahn equation \cite{DuZhu_2005, Ju_JSC2015, Xu_SIAM2006, Yang_JCP2007, WangJuDu_JCP16} 
to rewrite $W'(\phi) = 36(\phi^2-\phi)(2\phi-1) $ as 
\begin{align*}
W'(\phi) = \kappa\phi + \left[  W'(\phi) - \kappa \phi \right], 
\end{align*}
where $\kappa\ge0$ satisfying
\begin{align*}
\kappa \ge \dfrac{1}{2}\max\{ 0, \max_{0\le\phi\le1} W''(\phi) \} = 18.
\end{align*}
Similarly, we rewrite $f'(\phi)$ as 
\begin{align*}
f'(\phi) = \mu \phi + \left( f'(\phi) - \mu\phi \right),
\end{align*}
where $\mu\ge0$ satisfies
\begin{align*}
\mu \ge \dfrac{1}{2}\max\{ 0, \max_{0\le\phi\le1} f''(\phi) \} = 4.
\end{align*}

Note that the potentials $U_{\text{vdW}} $ and $U_{\text{ele}}$ are unbounded near $\textbf{x}_i$ 
for each $\bx_i\in \Omega_{\rm m}$.  Since the equilibrium phase field $\phi$ is expected to 
vanish in a small neighborhood of $\bx_i$ for each $i$, we truncate these
potentials with a numerical parameter $r_{\rm cut} > 0$. 
The truncated potential $U_{\rm vdW} $ is the sum of the 
truncated Lennard-Jones potentials
$U^{(i)}_{\rm LJ, cut}(r) $, defined by $U^{(i)}_{\rm LJ, cut}(r) = U^{(i)}(r)$
if $r \ge r_{\rm cut}$ and $U^{(i)}_{\rm LJ, cut}(r) = U^{(i)}_{\rm LJ} (r_{\rm cut} )$ otherwise. 
Similarly, we can truncate $U_{\rm ele}$ by modifying 
$(\bx - \bx_i) / | \bx - \bx_i |^3$ to $V^{(i)}(|\bx-\bx_i|) (\bx-\bx_i)/|\bx-\bx_i|^2$ for each $i$, 
where $V^{(i)} (r)  = 1/ r $ if $r \ge r_{\rm cut}$ and 
$V^{(i)} (r) = 1/r_{\rm cut}$ otherwise. 

For simplicity, let us still denote these modified potentials by 
$U_{\rm vdW} $ and $U_{\rm ele}$, respectively. Let us set 
\[
\nu = \sup_{x\in\Omega} |\rho_{\text{w}}U_{\text{vdW}}+U_{\text{ele}} |.
\]

Then the equation \reff{General_gradientflow} in a stabilized form reads
\begin{align}\label{eqn:gradientflow}
\partial_t\phi &= \left[\gamma\left( \epsilon\Delta\phi - \dfrac{\kappa}{\epsilon}\phi\right) - \mu\nu\phi \right] + \Big[ -\dfrac{\gamma}{\epsilon}\left(W'(\phi) - \kappa\phi\right)  - f'(\phi)(\rho_{\text{w}}U_{\text{vdW}} + U_{\text{ele}}) + \mu\nu\phi\Big] \nonumber \\
& = \mathcal{L}(\phi) + \mathcal{N}(\phi),
\end{align}
where the linear term is
\begin{align*}
 \mathcal{L}(\phi)  = \gamma\left( \epsilon\Delta\phi - \dfrac{\kappa}{\epsilon}\phi\right) - \mu\nu\phi,
\end{align*}
and the nonlinear term is
\begin{align*}
\mathcal{N}(\phi)  =  -\dfrac{\gamma}{\epsilon}\left(W'(\phi) - \kappa\phi\right) - f'(\phi)(\rho_{\text{w}}U_{\text{vdW}} + U_{\text{ele}}) + \mu\nu\phi.
\end{align*}
The new reformulation \reff{eqn:gradientflow} will be used for the time-discretization based on the exponential 
time differentiation (ETD)  Runge-Kutta method (ETDRK).

\subsection{Spectral spatial discretization under periodic boundary condition}

We consider a rectangular system domain $\Omega\subset\mathbb{R}^3$
\[
\Omega = \{-L_x<x<L_x, -L_y<y<L_y, -L_z<z<L_z\}
\]
for some positive numbers $L_x$, $L_y$, and $L_z$,
and impose the periodic boundary condition. We discretize $\Omega$ by a rectangular mesh which is uniform in each direction as follows:
\[
\bx_{ijk} = (x_i,y_j,z_k) = (-L_x + i h_x, -L_y + j h_y,-L_z + k h_z)
\]
for $0\le i \le N_x$, $0\le j \le N_y$, and $0\le k \le N_z$; $h_x = 2L_x/N_x$, $h_y = 2L_y/N_y$, and $h_z = 2L_z/N_z$. We choose a time step $\Delta t > 0$ and set $t_n = n \Delta t$. 

Let $\phi_{ijk}^{(n)}\approx \phi(x_i,y_j, z_k, t_n) = \phi(\bx_{ijk}, t_n)$ denote the approximate solution at grid $\bx_{ijk}$ and time $t_n$. Denote the approximate solution in array form as $ \Phi = (\phi_{ijk})_{0:N_x-1, 0:N_y-1, 0:N_z-1}, $and denote its discrete Fourier transform (DFT) by $ \hat{\Phi} = (\hat{\phi}_{ijk})_{0:N_x-1, 0:N_y-1, 0:N_z-1}. $ Notice that  the Laplacian operator $\Delta$ in the spectral space corresponds to the spectrum
\begin{align*}
\lambda_{ijk} = -\lambda_x^2(i) - \lambda_y^2(j) - \lambda_z^2(k), 
\end{align*}
where
\begin{align*}
& \lambda_x(i) = 
\begin{cases}
\pi i/L_x & \mbox{if } 0\le i \le N_x/2, \\
\pi (N_x-i)/L_x & \mbox{if } N_x/2\le i \le N_x-1, 
\end{cases}\\
& \lambda_y(j) = 
\begin{cases}
\pi j/L_y & \mbox{if } 0\le j \le N_y/2,  \\
\pi (N_y-j)/L_y & \mbox{if }  N_y/2\le j \le N_y-1,  
\end{cases}\\
& \lambda_z(k) = 
\begin{cases}
\pi k/L_z & \mbox{if } 0\le k \le N_z/2, \\
\pi (N_z-k)/L_z & \mbox{if } N_z/2\le k \le N_z-1.  
\end{cases}
\end{align*}
Taking the fast Fourier transform (FFT) \cite{Loan_1992} on both sides of the equation
 (\ref{eqn:gradientflow}) yields now
\begin{align}\label{eqn:phihat}
\hat{\Phi}_t = \mathbf{L}\odot\hat{\Phi} + \widehat{\mathcal{N}(\Phi)}, 
\end{align}
 where $\mathbf{L}\odot\hat{\Phi}$ is the FFT of $\mathcal{L}(\phi)$ and is given by 
\begin{align*}
&\mathbf{L}\odot\hat{\Phi} = (l_{ijk}\hat{\phi}_{ijk})_{0:N_x-1, 0:N_y-1, 0:N_z-1},\\
& l_{ijk} = \gamma\left(\epsilon\lambda_{ijk}-\dfrac{\kappa}{\epsilon}\right) - \mu\nu. 
\end{align*}
Note that, since  $\gamma, \epsilon, \kappa, \mu$ and $\nu$ are all positive, and $\lambda_{ijk}\le0$, we have $l_{ijk}<0$. Therefore the following point-wise version of (\ref{eqn:phihat}) is asymptotically stable:
\begin{align}\label{eqn:phihat_pointwise}
\partial_t \hat{\phi}_{ijk}  = l_{ijk}\hat{\phi}_{ijk} + \left[\widehat{\mathcal{N}(\Phi)}\right]_{ijk},\quad 0\le i\le N_x-1,\  0\le j \le N_y-1,\  0\le k\le N_z-1. 
\end{align}
We will develop next high-order Runge-Kutta approximations based on the exponential time differencing for the time integration of (\ref{eqn:phihat_pointwise}).

\subsection{Exponential time differencing Runge-Kutta approximations}

In this section, we adopt the exponential time differencing (ETD) method \cite{Cox_JCP2002, Kassam_SIAMSC2005, JuZhangDu_CMS15, WangJuDu_JCP16} to explicitly and accurately solve the semi-discrete system (\ref{eqn:phihat}) or (\ref{eqn:phihat_pointwise}). Let $\Delta t_n$ be the time step size at time $t_n$: $t_{n+1} = t_n + \Delta t_n$. Integrating the equation (\ref{eqn:phihat_pointwise}) over a single time step from $t_n$ to $t_{n+1}$ yields
\begin{align}\label{eqn:ETD}
\hat{\phi}_{ijk}(t_{n+1}) = e^{l_{ijk}\Delta t_n} \hat{\phi}_{ijk}(t_n) +  e^{l_{ijk}\Delta t_n} \int_0^{\Delta t_n} e^{-l_{ijk}\tau} \left[ \widehat{\mathcal{N}(\Phi})(t_n+\tau) \right]_{ijk}\ d\tau, 
\end{align}
which is exact. We apply  various ETD-based methods to this equation as follows: approximate the nonlinear part $[\widehat{\mathcal{N}(\Phi)}]_{ijk}$ by polynomial interpolations and then perform exact integrations on the new integrands \cite{Ju_JSC2015, Cox_JCP2002}.

Denote by $\hat{\Phi}^n = (\hat{\phi}^n_{ijk})$ the numerical approximation of $\hat{\Phi}(t_n) = (\hat{\phi}_{ijk}(t_n))$.  Then the first-order scheme by the ETD Euler approximation, ETD1 (or ETD1RK), is given by
\begin{align*}
&\hat{\Phi}^{n+1} = \text{ETD1RK}(\hat{\Phi}^{n}, \Delta t_n, \mathcal{L}, \mathcal{N}):  \\
& \hat{\phi}^{n+1}_{ijk} = e^{l_{ijk}\Delta t_n}\hat{\phi}^{n}_{ijk} + l^{-1}_{ijk}(e^{l_{ijk}\Delta t_n}-1) \left[\widehat{\mathcal{N}(\Phi^n)}\right]_{ijk} .
\end{align*}
Higher-order ETD schemes can be constructed based on multi-step or Runge-Kutta approximations. The 2nd, 3rd and 4th order Runge-Kutta schemes, which we refer as ETD2RK, ETD3RK, and ETD4RK, respectively, can be found in \cite{Cox_JCP2002}. For the equation (\ref{eqn:phihat}) we have the 2nd order scheme (ETD2RK):
\begin{align*}
\hat{\Phi}^{n+1} &= \text{ETD2RK}(\hat{\Phi}^{n}, \Delta t_n, \mathcal{L}, \mathcal{N}):\\
&
\begin{cases}
\mathbf{A} = (a_{ijk}) = \text{ETD1RK}(\hat{\Phi}^{n}, \Delta t_n, \mathcal{L}, \mathcal{N}), \\
\hat{\phi}^{n+1}_{ijk} = a_{ijk} + \Delta t_n^{-1}l^{-2}_{ijk}(e^{l_{ijk}\Delta t_n} - 1 - l_{ijk}\Delta t_n) \left[ \widehat{\mathcal{N}(\check{\mathbf{A}}) }  - \widehat{\mathcal{N}(\Phi^n)}  \right]_{ijk}, 
\end{cases}
\end{align*}
where $\check{\mathbf{A}}$ stands for the inverse discrete Fourier transform (iDFT) of $\mathbf{A}$. The 4th order scheme (ETD4RK) reads
\begin{align*}
\hat{\Phi}^{n+1} &= \text{ETD4RK}(\hat{\Phi}^{n}, \Delta t_n, \mathcal{L}, \mathcal{N}):\\
&
\begin{cases}
\mathbf{A} = (a_{ijk}) = \text{ETD1RK}((\hat{\Phi}^{n}, \Delta t_n/2, \mathcal{L}, \mathcal{N}), \\
\mathbf{B} = (b_{ijk}) = e^{l_{ijk}\Delta t_n/2}\hat{\phi}_{ijk}^{n} + l^{-1}_{ijk} (e^{l_{ijk}\Delta t_n/2}-1) \left[ \widehat{\mathcal{N}(\check{\mathbf{A}}) } \right]_{ijk}, \\
\mathbf{C} = (c_{ijk}) = e^{l_{ijk}\Delta t_n/2}a_{ijk} + l^{-1}_{ijk} (e^{l_{ijk}\Delta t_n/2}-1) \left[ 2\widehat{\mathcal{N}(\check{\mathbf{B}})} - \widehat{\mathcal{N}(\Phi^n)} \right]_{ijk},  \\
\hat{\phi}^{n+1}_{ijk} = e^{l_{ijk}\Delta t_n}\hat{\phi}_{ijk}^{n}  + \Delta t_n^{-2}l^{-3}_{ijk}\times\\
\qquad\qquad \Bigg\{ \bigg( -4 - l_{ijk}\Delta t_n + e^{l_{ijk}\Delta t_n}(4-3l_{ijk}\Delta t_n + l^2_{ijk}\Delta t_n^2)  \bigg) \left[ \widehat{\mathcal{N}(\Phi^n)} \right]_{ijk}  \\
\qquad\qquad\qquad +  2\bigg( 2 + l_{ijk}\Delta t_n + e^{l_{ijk}\Delta t_n}(-2+l_{ijk}\Delta t_n)  \bigg) \left[ \widehat{\mathcal{N}(\check{\mathbf{A}})} +  \widehat{\mathcal{N}(\check{\mathbf{B}})} \right]_{ijk}\\
\qquad\qquad\qquad +  \bigg( -4 - 3l_{ijk}\Delta t_n - l^2_{ijk}\Delta t_n^2 + e^{l_{ijk}\Delta t_n}(4-l_{ijk}\Delta t_n)  \bigg) \left[ \widehat{\mathcal{N}(\check{\mathbf{C}})} \right]_{ijk} \Bigg\} .
\end{cases}
\end{align*}

\section{Numerical Tests and Applications}\label{s:Tests}

In this section,  we first  validate our theory, particularly the incorporation of the new term $f(\phi)$ in the gradient-flow dynamics \reff{General_gradientflow}, by comparing it to the old model (\ref{f03}) for a one-particle system. For reference, a table of parameter values is listed in Table \ref{table:parameters}. We then apply our ETD-based Runge-Kutta method to a two-plate system. We compare the ETD1RK, ETD2RK and ETD4RK for the numerical efficiency, and the corresponding convergence rates. Then for different distances of separation of the two parallel plates with various charge combinations, we calculate the different components of the mean-field free-energy with loose and tight initial surfaces.

\begin{table}[hbtp]
\begin{center}
{\small 
\begin{tabular}{ | l | p{7.5cm} | }
  \hline
   $P = 0\ pN/${\AA}$^2$ & Pressure  \\
  $T = 300\ \text{K}$ & Tempature  \\
  $\gamma_0 = 0.175\ k_{\text{B}}T/${\AA}$^2$ & Surface tension  \\
  $\rho_{\text{w}} = 0.0333$ {\AA}$^{-3}$ & the constant solvent (water) density. \\
  $\varepsilon_i = \varepsilon_{\text{LJ}} = 0.3\ k_{\text{B}}T,\ i = 1:N$ & the depth of the Lennard-Jones potential well associated with the $i$th solute atom.  \\
  $\sigma_i = \sigma_{\text{LJ}} = 3.5$ {\AA},  $\ i = 1:N$ & the finite distance at which the Lennard-Jones potential of $i$th solute atom is zero. \\
  $r_{\text{cut}} = 0.7\sigma_{\text{LJ}}$ & the radius of truncation for potential \\
  $\varepsilon_0 = 1.4321\times 10^{-4}\ \text{e}^2/(k_{\text{B}}T${\AA}$)$ & vacuum permittivity \\
  $\varepsilon_{\text{m}} = 1$ & relative permittivity of the solute \\
  $\varepsilon_{\text{w}} = 80$ & relative permittivity of the solvent (water) \\
  $Q_i$ in units e & partial charge of the $i$th solute atom at $\textbf{x}_i$, which vary in different examples. \\
   $\epsilon$ in units {\AA} & the interfacial width of the phase field $\phi$, which  vary in different examples \\
  \hline
\end{tabular}
}
\caption{\small Parameters in the model.}
\label{table:parameters}
\end{center}
\end{table}

\vspace{-10 mm}

\subsection{One-particle system}

We now validate our theory by considering a one-particle system $(N = 1)$. We place a single point charge $Q$ at the origin immersed in water.  As the one-particle system is radially symmetric, the phase-field free-energy functional \reff{General_phasefield} reduces to that of radially symmetric phase fields $\phi = \phi(r)$ ($N = 1$ and $Q_1 = Q$): 
\begin{align}\label{Fone}
F^{\epsilon,\text{rad}}[\phi] &= 4 \pi \gamma_0 \int_0^\infty 
\left[ \frac{\epsilon}{2} |\phi'(r)|^2+\frac{1}{\epsilon} W(\phi(r))\right] r^2\, dr \nonumber\\
& + 4\pi\rho_{\text{w}} \int_0^\infty f(\phi)  U_{\text{vdW}}(r) r^2 \, dr 
 +  \dfrac{Q^2}{8\pi\varepsilon_0}\left(\dfrac{1}{\varepsilon_{\text{w}}} - \dfrac{1}{\varepsilon_{\text{m}}}\right) 
\int_0^\infty f(\phi)/r^2 dr,  
\end{align}
where $U_{\text{vdW}}(r)$ is given by \reff{UvdW} with $N = 1$, $\bX = \textbf{0}$, 
and  $\ve_1$  and $\sigma_1$ are given in Table \ref{table:parameters}.

Taking $Q = 2 e$, $\epsilon = 0.1${\AA}, computational domain $= [0, 5]$, $\Delta x = 5\times 10^{-4}$ and $\Delta t = 10^{-6}$,  and other parameter values from Table \ref{table:parameters}, we solve the gradient-flow dynamics $\partial_t\phi = -\delta F^{\epsilon,\text{rad}}[\phi]/\delta \phi$. The numerical scheme we adapt here is the Crank--Nicolson method \cite{LeVeque_2007} and Thomas algorithm \cite{Higham_2002} for the corresponding tri-diagonal linear system. 

\begin{figure}[htbp]
\centerline{
\includegraphics[width=200mm]{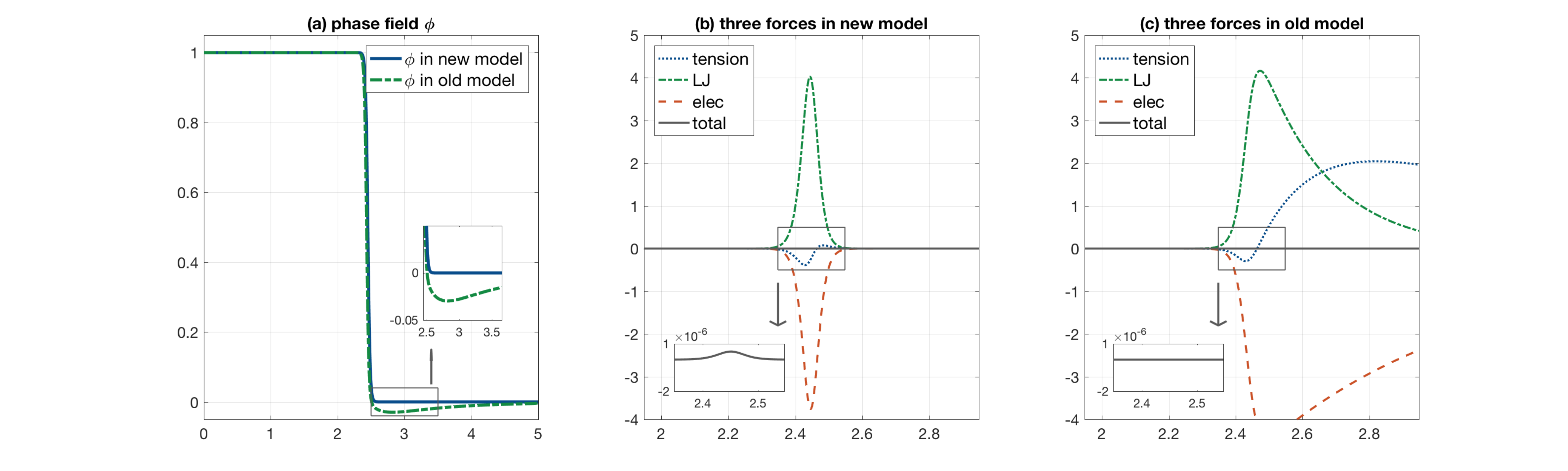}
$ \quad $ }
\vspace{-4 mm}
\caption{\small Numerical comparison between the new model \reff{f_term} and the old one \reff{f03} for the one-particle system. (a) The two phase-field functions $\phi$ at $t=10$ in which the $\phi$ of new model presents a desirable
hyperbolic tangent   profile, but the $\phi$ of old model displays a deviation of $O(10^{-2})$ from 0 as seen in the inset.  (b) The three forces in the new model (surface tension, Lennard-Jones force, and the electrostatic force) are
localized only near the interface and the sum is balanced up to $O(10^{-6})$.  (c) The three forces in the old model make nonzero contributions in the region of $\{\phi\approx 0\}$. All the three subfigures are plotted near the interface $[R_0-0.5, R_0+0.5]$, where $R_0=2.4479$ is determined numerically by $\phi(R_0)=0.5$ using the new model.}
\label{fig:OneParticle}
\end{figure}

Figure {\ref{fig:OneParticle}} presents the numerical comparison between our new model 
$f(\phi) = (\phi^2-1)^2$ and the old model $f(\phi) = (\phi-1)^2$. 
Our new phase-field implementation improves the old ones in several aspects. 
(1) The new model displays a better hyperbolic tangent profile than the old one as seen in Figure 2 (a). More specifically, the equilibrium phase field $\phi$ in the new model shows a desirable hyperbolic tangent shape which monotonically changes its value from 1 to 0, while the old model presents some unphysical feature near the interfacial region, where $\phi$ has a deviation of $O(0.01)${\AA} away from $0$ and takes negative values. 
(2) The new model maintains the force localization near the interface as seen in Figure 2 (b). In the old model, all the three forces have nonzero contributions in the region of $\{\phi\approx 0\}$.
(3) The force localization 
in the new model allows us to use a much smaller computational box that encloses the entire solute region and solute-solvent interface which greatly improves the computational efficiency.  Of course, the deviation of $\phi$ can be mitigated by letting $\epsilon\rightarrow 0$ by the theoretical study in \cite{LiZhao_SIAP13}. However, in real applications, especially in the 3D simulations, $\epsilon$ has to remain relatively large to reduce the computational cost. Therefore, the new model is advantageous for keeping the hyperbolic tangent profile of $\phi$ and localizing  the forces only near the interfaces even for a relatively large $\epsilon$. It is worth mentioning that the force localization due to $f(\phi)=(\phi^2-1)^2$ occurs not only at the equilibrium, but in the entire gradient-flow dynamics. Therefore it can potentially be used to study non-equilibrium dynamics such as  cell motion \cite{CamleyZhao_PRE_2017}. 

\begin{table}
\begin{center}
{\small 
\begin{tabular}{|c|c|c|c|c|c|c|cl}
\hline
 $Q$ & Optimal Radii/Energy & $\epsilon$ = 0.5 & $\epsilon$ = 0.2 & $\epsilon$ = 0.05 & $\epsilon$ = 0.02 & $\epsilon=0$ 
 \\
\hline
        & $R_{\min}$        & 3.08013 & 3.06058 & 3.055 &  3.05411 & 3.054 \\
        & $F_{\rm surf}$   &20.90351 & 20.60341 &  20.514 & 20.50996 & 20.511  \\
0.0   & $F_{\rm vdW}$  &-2.55793 & -2.61359 & -2.627 &  -2.63751 & -2.644\\
        & $F_{\rm elec}$   &0.00000  & 0.00000  & 0.000 &  0.00000  & 0.000     \\
        & $F_{\rm tot}$      &18.34557 & 17.98982 &  17.887 & 17.87245 & 17.867 \\
\hline
       & $R_{\min}$         & 2.987 & 2.967 & 2.961 &  2.960     & 2.960 \\
       & $F_{\rm surf}$    &19.672 &  19.366 &  19.275 & 19.266 & 19.267  \\
0.5  & $F_{\rm vdW}$   &-0.980 & -1.025 & -1.036 &  -1.042  & -1.054\\
       & $F_{\rm elec}$    &-23.080 & -23.162 & -23.177 &  -23.177 & -23.173 \\
       & $F_{\rm tot}$       &-4.388 & -4.822 &  -4.938 & -4.953    & -4.960 \\
\hline
       & $R_{\min}$         &  2.79823  &   2.77930   & 2.77252    &  2.77154   & 2.771 \\
       & $F_{\rm surf}$    & 17.32496  & 16.99413  & 16.90424 & 16.89034  & 16.886  \\
1.0  & $F_{\rm vdW}$   &   5.10415  &   5.11240  &   5.11524 &   5.11501   &  5.113  \\
       & $F_{\rm elec}$    &-98.54247  &-98.92329 &-99.00642 &-99.01096 & -99.012 \\
       & $F_{\rm tot}$       &-76.11335  &-76.81676 &-76.9869 &-77.00560 &-77.014 \\
\hline
       & $R_{\min}$          &  2.61690  &   2.60079   & 2.59418   &  2.59318      & 2.593 \\
       & $F_{\rm surf}$     & 15.31472  & 14.89081  & 14.79960 & 14.78639      & 14.782  \\
1.5  & $F_{\rm vdW}$    & 17.83743   & 17.95046  & 17.96966   &  17.97163    & 17.971\\
       & $F_{\rm elec}$     &-236.98862  &-237.86930 &-238.08700 &-238.10064 & -238.105 \\
       & $F_{\rm tot}$       &-203.83648   &-205.02804   &-205.31774 &-205.34262 & -205.354 \\
\hline
         & $R_{\min}$        &  2.46839  &   2.456   & 2.44947    &  2.44851      & 2.448 \\
         & $F_{\rm surf}$   & 13.94052  & 13.304  &  13.19387 & 13.18262      & 13.178  \\
2.0    & $F_{\rm vdW}$  & 38.47104  & 38.676  & 38.76414   &  38.75819    & 38.757\\
         & $F_{\rm elec}$   & -446.41599 & -447.827 & -448.28042 &  -448.30575 & -448.317 \\
         & $F_{\rm tot}$     & -394.00443   & -395.848   & -396.32242 & -396.36494  & -396.381 \\
\hline
\end{tabular}
}
\caption{\small A comparison of numerical results obtained by the phase-field calculations
(solving the gradient-flow dynamics  (\ref{Fone}))
and by the sharp-interface calculations (minimizing numerically the function 
$G[R]$ in \reff{GR}) for the solvation of a single-particle system. The 
sharp-interface (indicated with $\epsilon = 0$) results are presented in the last column. See the text for the units.}
\label{table:oneparticle}
\end{center}
\end{table}

We now compare our results of phase-field computations  with those of the sharp-interface implementation. For a one-particle system, the sharp-interface free-energy functional \reff{General_sharpinterface2} is a one-variable function of the radius $R$ of the solute sphere centered at the origin \cite{WangEtal_VISMCFA_JCTC12} 
\begin{align}\label{GR}
F[\Gamma]: = F[R] = 4 \pi \gamma_0  R^2 
 + 16 \pi \rho_{\rm w} \ve \left( \frac{ \sigma^{12}}{9R^9}
- \frac{\sigma^{6}}{3R^3}  \right) 
 + \frac{Q^2}{8 \pi \ve_0 R} \left( \frac{1}{\ve_{\rm w}} - \frac{1}{\ve_{\rm m}}\right). 
\end{align}
This one-variable function can be minimized numerically with a very high accuracy. 

We test on a set of $Q$-values: $Q = 0.0e, 0.5e, 1.0e, 1.5e, 2.0e. $ We use both the sharp-interface and phase-field models to calculate the optimal radius $R_{\rm min}$, the total minimum free energy $F_{\rm tot}$, and the corresponding surface energy $F_{\rm surf}$, solute-solvent van der Waals interaction energy $F_{\rm vdW}$,  and the electrostatic energy $F_{\rm elec}$, respectively.  For our phase-field calculations, we use different values of  the numerical parameter $\epsilon$.   Table~\ref{table:oneparticle} shows our computational results.  It is clear that as $\epsilon$ becomes smaller, the result of the phase-field model is also closer to that of the sharp-interface model.

\subsection{Two parallel plates}

We now consider the system of two parallel molecular plates that has been studied by the molecular dynamics simulations \cite{KoishiEtal_PRL2004} and by the sharp-interface VISM \cite{WangEtal_VISMCFA_JCTC12}. Each plate consists of $N_p \times N_p$ fixed CH$_2$ atoms with $N_p = 6$ and the atom-to-atom distance $d_0 = 2.1945$ {\AA}.   The plate has a square length of about $30$ {\AA}. The two plates are placed in parallel with a center-to-center distance $d$. We use the parameter values  listed in Table \ref{table:parameters}. To study the charge effect, as in \cite{WangEtal_VISMCFA_JCTC12}, we assign central charges $q_1$ and $q_2$ to the first and second plates, respectively, with $|q_1| = | q_2|$.  The total charges of these two plates are $36 q_1 $ and $ 36 q_2,$ respectively.  

Let us consider the gradient-flow dynamics (\ref{eqn:gradientflow}) starting with two parallel plates of separation $d = 12${\AA}. We choose the uniform spatial mesh $256^3$ with $L_x=L_y = L_z = 18${\AA} (i.e., the mesh size $h = 2L_x/256$) and set $\epsilon = 0.5$. The time step is taken uniformly as $\Delta t = 0.05$. We use two types
of initial phase-field functions. One is called a {\it loose initial}, such as the characteristic function of a box 
\[
\{ (x,y,z): |x|\le (N_p-1)d_0 + \sigma_{\text{LJ}},\ |y|\le \dfrac{d}{2}+\sigma_{\text{LJ}},\ |z|\le (N_p-1)d_0+\sigma_{\text{LJ} } \}
\]
that contains the two plates. The other is called a {\it tight initial}, which can be the characteristic function
of two boxes that wrap up the two plates separately. We set the stopping criteria for our time
iteration by 
\begin{align*}
\dfrac{F^{\epsilon}\big[\phi^{(n+1)}\big] - F^{\epsilon}\big[\phi^{(n)}\big] }{\Delta t_n} < \text{TOL} = 10^{-3}
\end{align*}

Figure \ref{fig:twoplates01} shows stable equilibrium solute-solvent surfaces of two-plate system obtained by solving the gradient-flow dynamics (\ref{eqn:gradientflow}) with loose  initials of separation $d_0 = 12${\AA}. The partial charges are $(q_1,q_2)=(0.1e,0.1e)$, $(-0.1e,0.1e)$, $(0.2e,0.2e)$, $(-0.2e,0.2e)$, respectively, from left to right. Note that the larger the partial charges are, the tighter the solute-solvent surfaces wrap the two plates. Meanwhile the surfaces wrap tighter when the partial charges change from $+/+$ to $-/+$.

Figure \ref{fig:twoplates_energy} shows the energy evolution for the gradient-flow dynamics of the two-plate system with loose  initial of separation $d_0 = 12${\AA} and $(q_1,q_2) = (0.2e,0.2e)$. The stabilized ETD1RK, ETD2RK, and ETD4RK schemes are adopted with different values of time step size $\Delta t = 1, 0.1$ and $0.01$. The first row compares the energy curves under different time step sizes for each of the three stabilized ETDRK schemes, while the second row reorganizes the curves using different schemes but with the same time step size. It is easy to see that all the schemes work stably with all time step sizes, and converge as the time step size is decreased. The lower right plot in Figure  \ref{fig:twoplates_energy}  shows that for $\Delta t = 0.01$ the energy curves for different schemes are nearly indistinguishable. A good agreement is also found between the curves for $\Delta t= 0.1$ and 
$\Delta t=0.01$ for ETD4RK in the upper right plot of Figure~\ref{fig:twoplates_energy}. 

\begin{figure}[htbp]
\centerline{
\includegraphics[width=120mm]{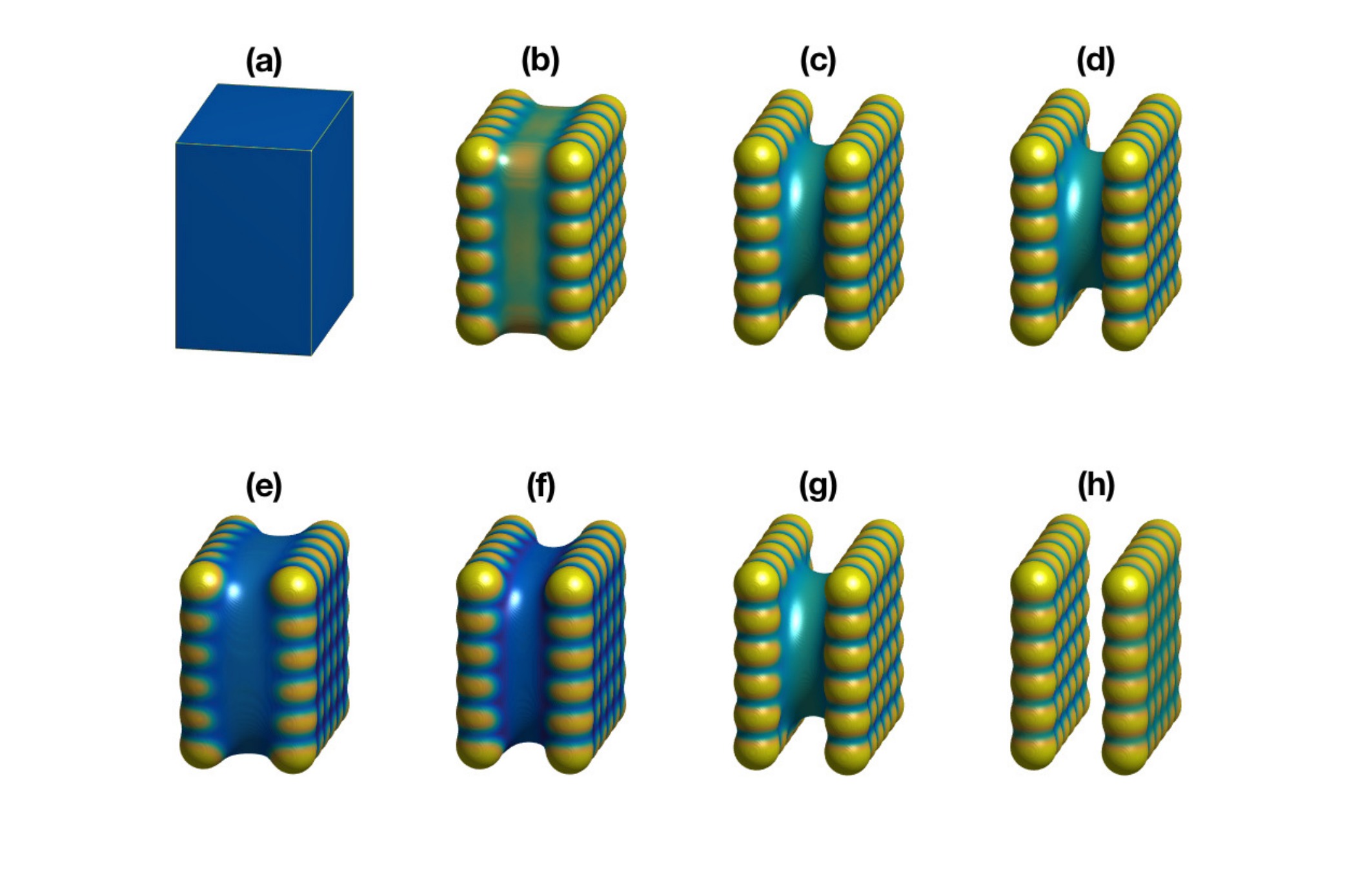}
}
\vspace{-0.4in}
\caption{\small Simulation of the gradient-flow dynamics (\ref{eqn:gradientflow}) for the two-plate system using the stabilized ETD4RK scheme with $\Delta t = 0.05$ and loose initials. The surfaces are defined as the 1/2-level set of a phase-field function $\phi$. The plate-plate separation is fixed to be $d = 12${\AA}. From (a) to (d): the snapshots at $t = 0, 50, 500$ and $1000$ during the gradient-flow dynamics with $(q_1,q_2)=(0.2e,0.2e).$  From (e) to (h), the equilibrium states of the solute-solvent interface for different partial charges $(q_1,q_2)=(0.1e,0.1e)$, $(-0.1e,0.1e)$, $(0.2e,0.2e)$, $(-0.2e,0.2e)$, respectively. }
\label{fig:twoplates01}
\end{figure}

\begin{figure}
\centerline{
\includegraphics[width=180mm]{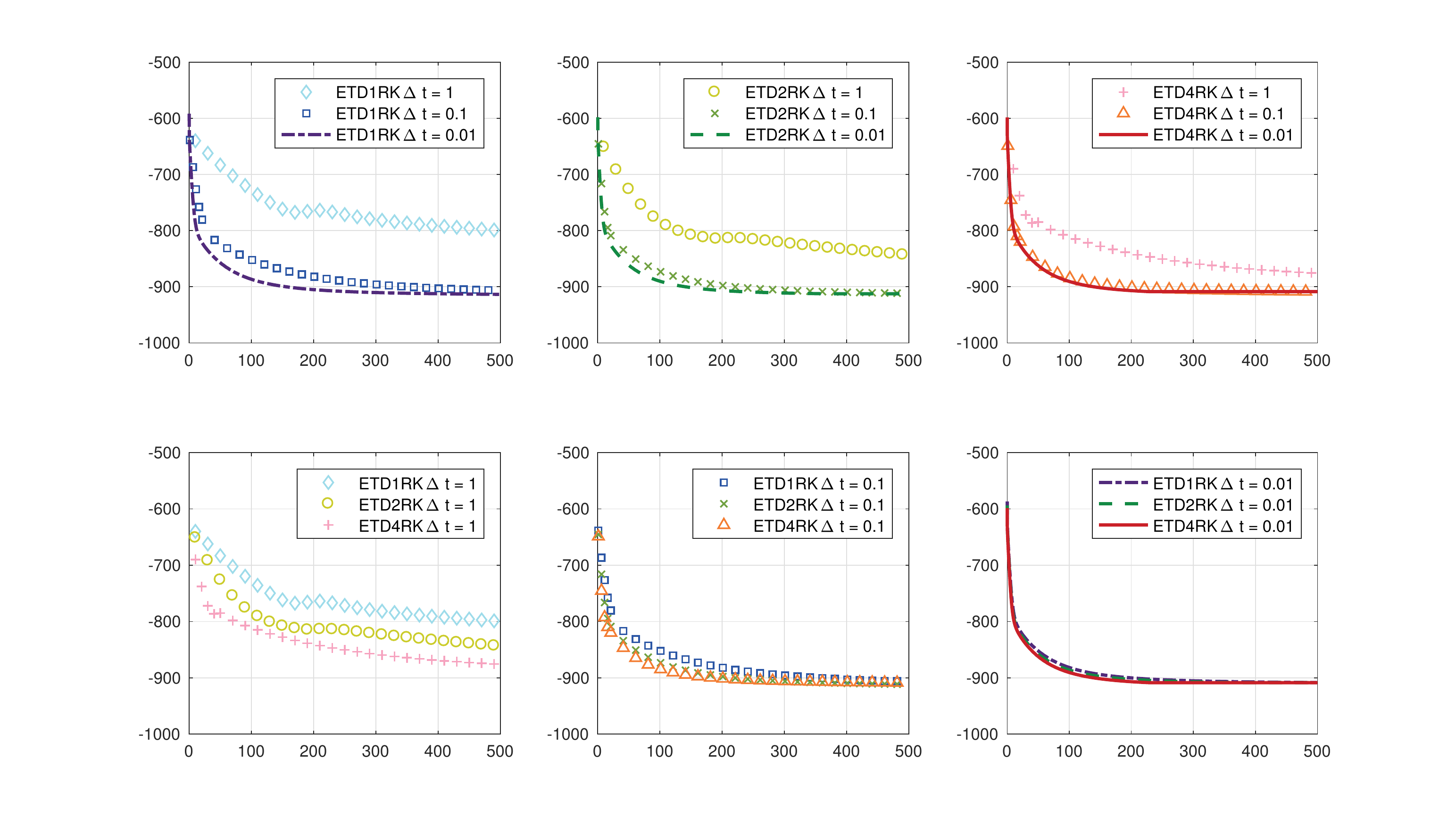}
}
\vspace{-6 mm}
\caption{\small Energy evolution for the gradient-flow dynamics (\ref{eqn:gradientflow}) for the two-plate system with loose initial of plate separation $d = 12${\AA}. }
\label{fig:twoplates_energy}
\end{figure}

We further test the convergence rates of the stabilized ETDRK schemes. To this end, we perform the simulations on a small time interval $[0,1]$. We take the solution generated by the ETD4RK scheme with $\Delta t = 10^{-4}$ as the benchmark solution and then compute the errors in energy for all schemes with larger step sizes. Table \ref{table:rate} presents the energies, errors and the convergence rates based on the data at $t=1$ for all schemes with time step sizes being halved from $\Delta t = 1\times 10^{-1}$ to $1.5625\times 10^{-3}$. These data are also used to generate Figure \ref{fig:convergencerate} which shows energy errors against time step sizes in a logarithmic plot for different ETD Runge-Kutta schemes. We can see from both the table and curves that the numerically computed convergence rates all tend to approach the theoretical values. Moreover, to obtain an energy error comparable to that of ETD1RK with $\Delta t = 1.5625\times 10^{-3}$, we can take a $2^3$-times larger step size for ETD2RK, or a $2^6$-times larger step size for ETD4RK. Since the computational cost of ETD4RK scheme is about 4 times of that for
ETD1RK per step, the ETD4RK scheme basically provides a factor of 16 speed-up at this particular accuracy level for this special test case.

\begin{table}
\scriptsize
\begin{center}
\begin{tabular}{lllllllllllllr}
\hline
$\Delta t$ &  \multicolumn{3}{l}{ETD1RK} & & \multicolumn{3}{l}{ETD2RK} & &  \multicolumn{3}{l}{ETD4RK} \\
\cline{2-4} \cline{6-8} \cline{10-12}
    & Energy & Error & Rate & & Energy & Error & Rate & & Energy & Error & Rate \\
\hline
$1.0000\times 10^{-1}$       & -640.023    & 14.594   & --     & & -646.0728 & 8.5448 & --      & & -653.93952183 & 3.1e-1 & -- \\
$5.0000\times 10^{-2}$       & -646.118    & 8.499   & 0.78 & & -651.7595 & 2.8580 & 1.58 & & -654.58950486 & 2.8e-2 & 3.48 \\
$2.5000\times 10^{-2}$      & -649.866    & 4.751   & 0.84 & & -653.6880 & 0.9295 & 1.62 & & -654.61527138 & 2.3e-3 & 3.58\\
$1.2500\times 10^{-2}$      & -652.094    & 2.522   & 0.91 & & -654.3495 & 0.2680 & 1.79 & & -654.61743360 & 1.8e-4 & 3.71 \\
$6.2500\times 10^{-3}$      & -653.316    & 1.301   & 0.95 & & -654.5453 & 0.0722 & 1.89 & & -654.61760092 & 1.3e-6 & 3.81 \\
$3.1250\times 10^{-3}$      & -653.956    & 0.661   & 0.98 & & -654.5987 & 0.0188 & 1.94 & & -654.61761288 & 9.1e-7 & 3.82 \\
$1.5625\times 10^{-3}$      & -654.284    & 0.333   & 0.99 & & -654.6127 & 0.0048 & 1.95 & & -654.61761373 & 6.0e-8 & 3.92\\
$10^{-4}$ (Benchmark)      & --    & --   & -- & & -- & -- & -- & & -654.61761379 & -- & -- \\
\hline
\end{tabular}
\caption{\small The energies, errors and the corresponding convergence rates at time $t=1$ by the stabilized ETD1RK, ETD2RK and ETD4RK schemes for the gradient-flow dynamics (\ref{eqn:gradientflow}) with $(q_1,q_2) = (0.2e, 0.2 e)$.}
\label{table:rate}
\end{center}
\end{table}

\begin{figure}[htbp]
\centerline{
\includegraphics[width=80mm]{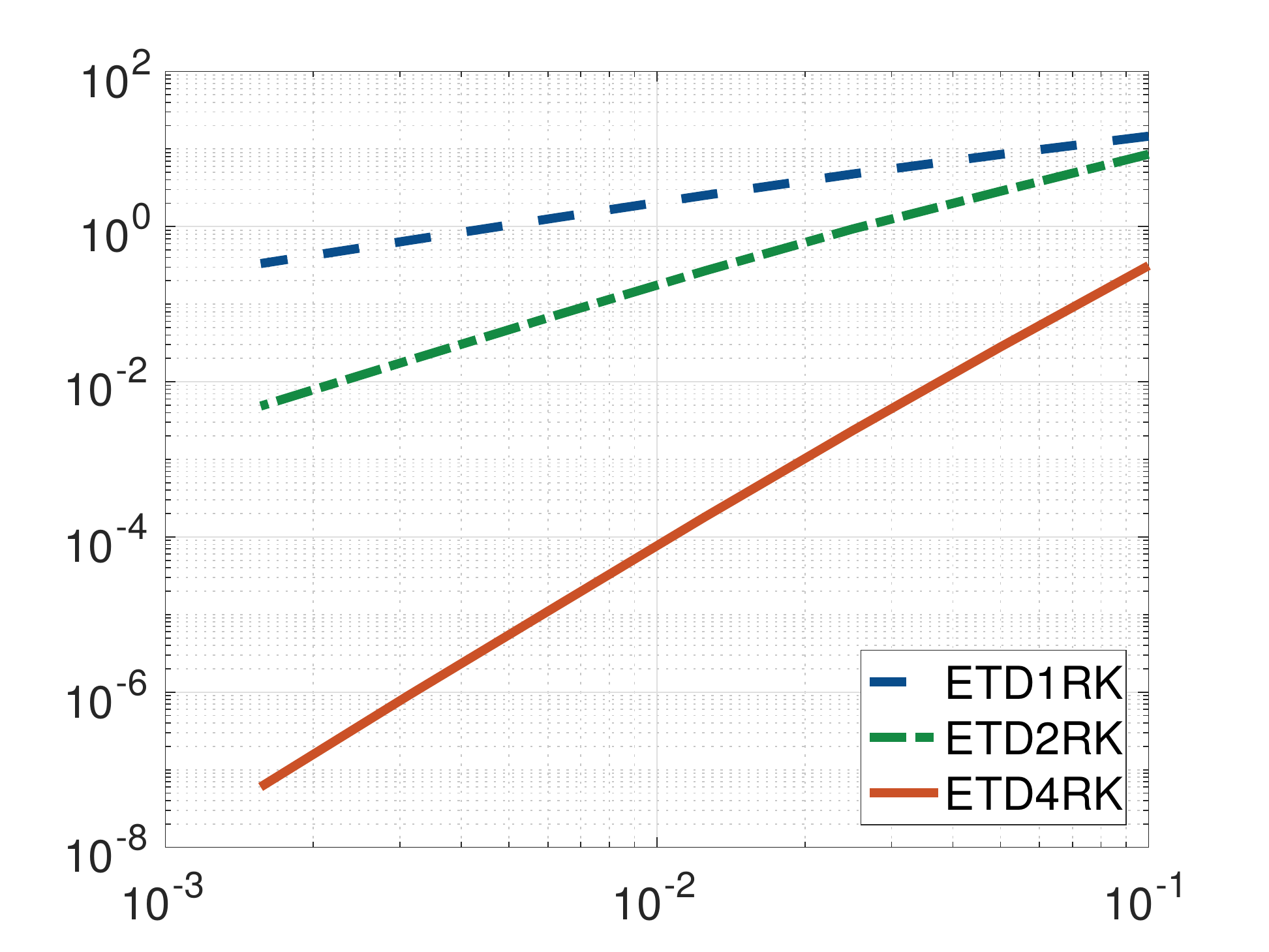}
}
\caption{\small Logarithmic plot of energy errors v.s. time step sizes of the stabilized ETD1RK, ETD2RK,  and ETD4RK schemes for the gradient-flow dynamics of two parallel plates with $(q_1,q_2)=(0.2e,0.2e)$.}
\label{fig:convergencerate}
\end{figure}

For a given reaction coordinate $d,$ there can be multiple stable equilibrium phase fields $\phi_d$ that are local minimizers of the phase-field VISM free-energy functional. In Appendix \ref{s:Appendix}, we briefly discuss the Potentials of Mean Force (PMF) which can effectively describe the solute-solute interaction. The PMF can have multiple branches along the reaction coordinate $d$, and hence can lead to hysteresis.  Strictly speaking, our PMFs are different from those defined using a Boltzmann average over all possible minimizers. Rather, our PMFs reflect possible branches of the VISM free energy along the reaction coordinate $d.$

In Figures~\ref{fig:twoplates_loose} and Figures~\ref{fig:twoplates_tight}, we plot the different components of the PMF with loose and tight initial surfaces, respectively. For the loose initials (Figure~\ref{fig:twoplates_loose}), the geometric part displays a strong attraction below a critical distance $d_c$ at which capillary evaporation begins. The crossover distance decreases from $d_c \simeq 21$~{\AA} for $(q_1,q_2) = $ ($-0.2\,$e, $+0.2\,$e) down to $9$~{\AA} for $(q_1,q_2) = $ ($0\,$e, $0\,$e). The value $21$ {\AA} is larger than $14$ {\AA} predicted by the sharp-interface VISM where the curvature correction was included. Note that the opposite charging has a much stronger effect than like-charging due to the electrostatic field distribution discussed above. Also the solute-solvent vdW part of the interaction is strongly affected by electrostatics due to the very different surface geometries induced by charging. Both curves $G_{\rm geo}^{\rm PMF}(d)$ and $G_{\rm vdW}^{\rm PMF}(d)$ demonstrate the strong sensitivity of nonpolar hydration to local electrostatics when capillary evaporation occurs and very ``soft" surfaces are present. For the surfaces resulting from the tight initials (Figure~\ref{fig:twoplates_tight}),  the situation is a bit less sensitive to electrostatics as the final surface is closer to the vdW surface for $d_c \gtrsim 6$~{\AA}.

\begin{figure}[hpbt]
\centerline{
\includegraphics[width=120mm]{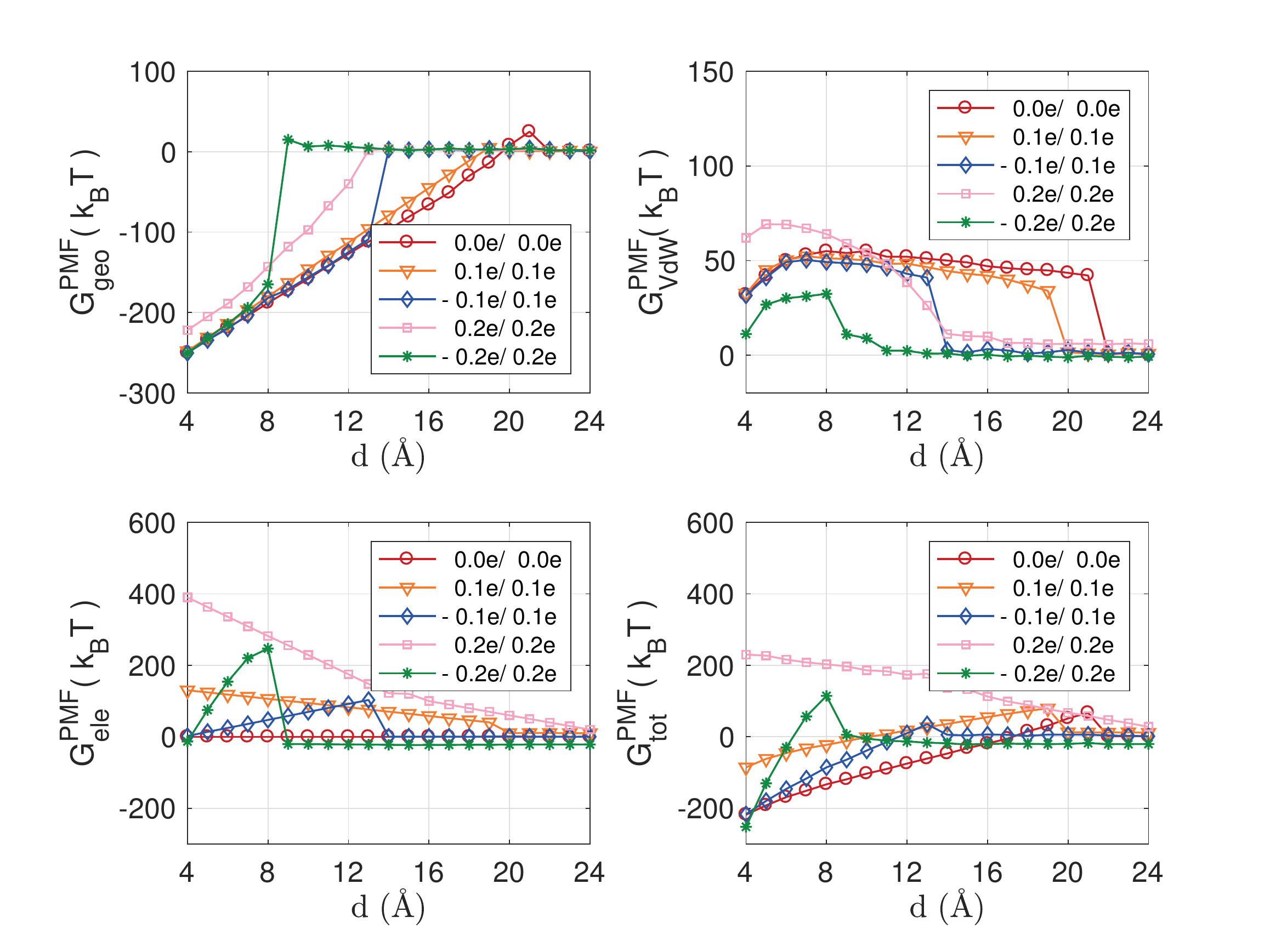}
}
\vspace{-6 mm}
\caption{\small Different components of the PMF for the two-plate system for different charge combinations $(q_1, q_2)$ (see legend) obtained by the phase-field VISM with loose initial surfaces.}
\label{fig:twoplates_loose}
\end{figure}

\begin{figure}[htbp]
\centerline{
\includegraphics[width=120mm]{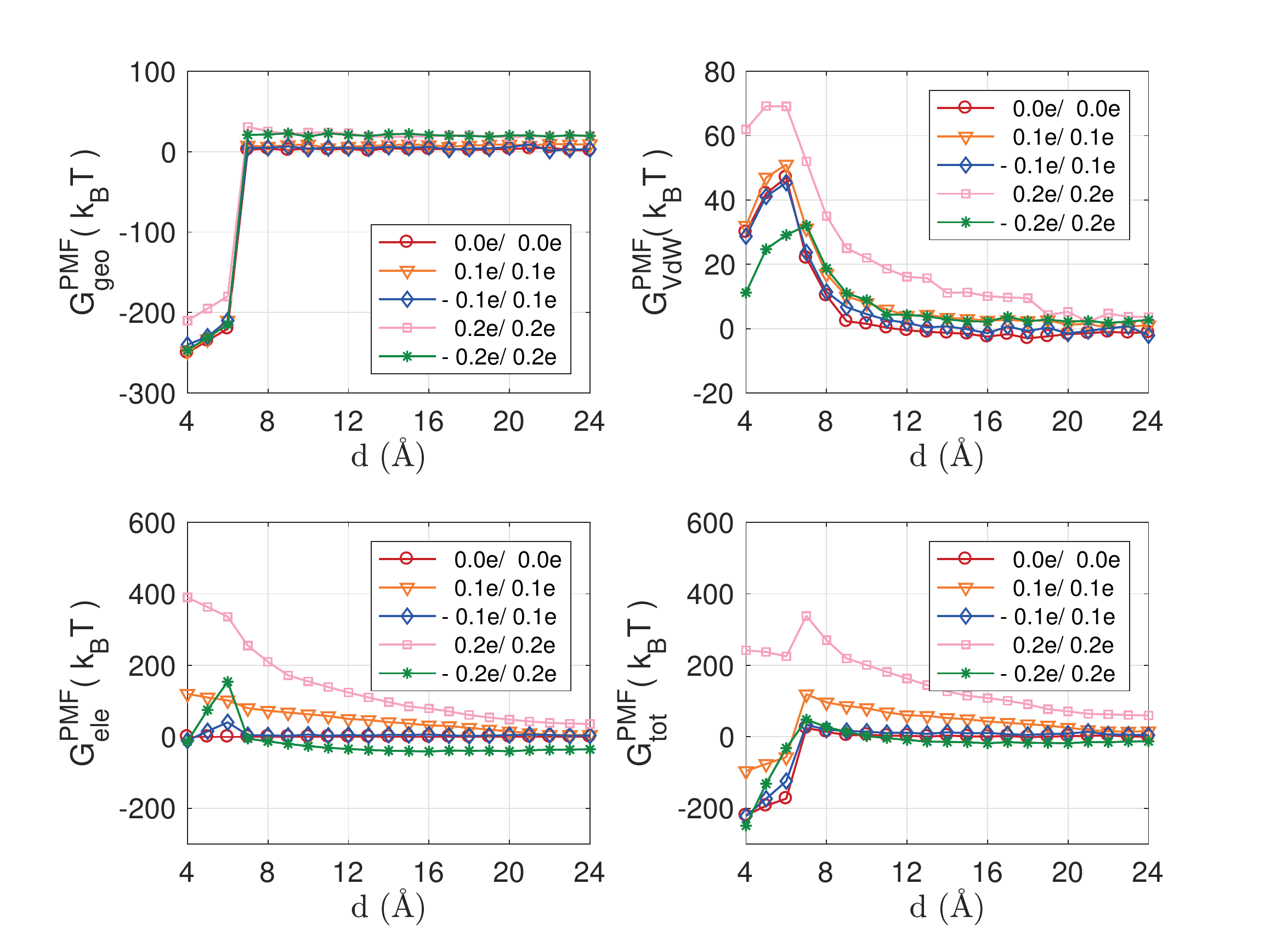}
}
\vspace{-3 mm}
\caption{\small Different components of the PMF for the two-plate system for different charge combinations $(q_1, q_2)$ (see legend) obtained by the phase-field VISM with tight initial surfaces.}
\label{fig:twoplates_tight}
\end{figure}


\section{Conclusions}

We have presented a new phase-field model to study the implicit solvation of charged molecules with Coulomb-field approximation. In this new model, we introduce the term $f(\phi) = (\phi^2-1)^2$ in \reff{General_phasefield} to localize the boundary force near the solute-solvent interface. In comparison with the old model used in our previous work, the new one keeps the force localized  only around
the interface. In addition, the new model displays a better hyperbolic tangent profile than the old one for a fixed interfacial width $\epsilon>0$. 

We have shown that our new phase-field model $\Gamma$-converges to the corresponding sharp interface model. To make our theory more general, we include the solute-solute mechanical interactions in the energy functional for our $\Gamma$-convergence
analysis. 

In developing the numerical method for the phase-field gradient-flow dynamics, we first adopt a linear splitting scheme to reformulate the underlying equation,  and then use an exponential time differencing method coupled with a Runge-Kutta scheme to solve the system which has been shown recently to be stable and efficient when dealing with a gradient-flow dynamics \cite{JuZhangDu_CMS15, WangJuDu_JCP16}. Using the two-plate system as a testing example, we have tested the efficiency and  convergence for the ETD1RK, ETD2RK, and ETD4RK schemes. Furthermore, we have used the ETD4RK scheme to study the effects of the separation of two plates and particle charges on the PMF. The simulations indicate that the two-plate system displays two different steady states obtained from loose and tight initials, respectively. The loose-initial steady state is energetically more favorable than the tight-initial steady state for a small distance of separation. When the distance of separation becomes larger and larger, the tight-initial steady state will becomes a more stable one. Our applications to single ions and two parallel charged plates have shown that our new theory and method can not only predict qualitatively well the solvation free energies for the system as in the previous studies \cite{DSM06a, Zhao_2013, SunWenZhao_JCP2015}, but more importantly improve the previous ones better in a few aspects such as maintaining desirable a 
hyperbolic tangent profile, keeping the force localized around the interface, and improving the computational efficiency by allowing a much smaller computational domain.

We are currently working to incorporate the Poisson--Boltzmann equation into our new phase-field VISM to better describe the electrostatic interaction. Another possible direction for our future study is to investigate the minimal energy path between the two solution branches of the two-plate system by coupling the phase-field VISM with the string method \cite{ERenVandenEijnden_PRB_2002, ERenVandenEijnden_JCP_2007, zhang2016recent} which will lead us the dynamics of two-plate system going from a loose-initial steady state to a tight-initial steady state.

\section*{Appendix}
\label{s:Appendix}
\renewcommand{\thesection}{A}
\setcounter{equation}{0}

To reduce the error in approximating the solute-solvent interaction energy caused by using a finite region $\Omega$, we replace the region of integral $\Omega$ in the last term in \reff{General_phasefield} by the entire space $\mathbb{R}^3$. Since the region outside $\Omega$ is filled with solvent where $\phi = 0$, this is equivalent to adding 
\begin{equation}
\label{EOmega}
\int_{\mathbb{R}^3\backslash\Omega} \rho U_{\text{vdW}}(\bx) + 
\int_{\mathbb{R}^3\backslash\Omega}  U_{\text{ele}}(\bx) \, d\bx.
\end{equation}

We now consider the potential of mean forces (PMF) for the two-plate system with the reaction 
coordinate being the plate-plate separation $d$ in {\AA}. Let us denote by 
$\phi_d$  a free-energy minimizing phase field corresponding to a given reaction coordinate $d.$  This phase-field function $\phi_d$ is a local
minimizer of the functional \reff{General_phasefield}, and $\phi_d = 0$ in $\R^3 \setminus \Omega.$ 
The total solvation free energy $F^{\epsilon}[\phi_d]$ is the sum of the geometrical part (the surface energy) $F^\epsilon_{\rm geo}[\phi_d]$,  the solute-solvent van der Waals interaction energy $F_{\rm vdW}[\phi_d]$,  and the electrostatic energy $F_{\rm ele}[\phi_d]:$
\[
F^{\epsilon}[\phi_d] = F^\epsilon_{\rm geo}[\phi_d] + F_{\rm vdW}[\phi_d] + F_{\rm ele}[\phi_d]. 
\]
These three terms are the same as those in \reff{General_phasefield}, except the integrals are over $\R^3.$ 
Since $\phi_d = 0$ outside $\Omega$, the first term $F^\epsilon_{\rm geo}[\phi_d]$ is exactly the same as
the first integral in \reff{General_phasefield} with $\phi_d $ replacing $\phi.$ 
As in \cite{WangEtal_VISMCFA_JCTC12,Che_VISM_JCTC13}, we define the (total) PMF by 
\begin{align*}
G^{{\rm PMF}, \epsilon}_{{\rm tot}}(d) 
= G_{\rm geo}^{{\rm PMF},\epsilon} (d)+G_{\rm vdW}^{\rm PMF} (d)+ G_{\rm ele}^{\rm PMF}(d),
\end{align*}
with 
\begin{align*}
& G_{\rm geo}^{{\rm PMF},\epsilon} (d) = F^\epsilon_{\rm geo}[\phi_d] - F^\epsilon_{\rm geo} [\phi_\infty],
\nonumber \\
& G_{\rm vdW}^{\rm PMF } (d) = F_{\rm vdW}[\phi_d] - F_{\rm vdW} [\phi_\infty] 
 +  \sum_{i\in \mbox{\tiny Plate I}}\, 
\sum_{j\in \mbox{\tiny Plate II}} U_{i,j}( |\bx_i-\bx_j| ), 
\\
& G_{\rm ele}^{\rm PMF}(d) = F_{\rm ele}[\phi_d]  - F_{\rm ele}[\phi_\infty] 
 + \frac{1}{4 \pi \ve_{\rm m} \ve_0}
\sum_{i \in \mbox{\tiny  Plate I} }\, \sum_{j \in \mbox{\tiny {Plate II}} } 
\frac{Q_iQ_j}{|\bx_i - \bx_j|}. 
\end{align*}
Here a quantity at $\infty$ is understood as the limit of that quantity at a coordinate $d'$ as $d' \to \infty,$ and $U_{i,j}$ is the Lennard-Jones interaction potential between $\bx_i$ and $\bx_j$. A quantity at $\infty$ can be calculated by doubling that of a single plate. 

For each $d$ and $\ve > 0$, we compute $\phi_d$ and $\phi_\infty$, the latter is obtained by minimizing \reff{General_phasefield} for a 
single plate. This is one of the two plates in terms of the solute atomic positions.
 Then, we can compute $G_{\rm geo}^{{\rm PMF}, \epsilon}(d) $ by evaluating integrals over $\Omega.$ 
The computation of $G_{\rm vdW}^{\rm PMF}(d) $ is similar, as both $F^\epsilon_{\rm geo}[\phi_d] $ and $ F^\epsilon_{\rm geo} [\phi_\infty]$
contain the first integral in \reff{EOmega}, so they cancel, and the calculation of double-sum term in $G_{\rm vdW}^{\rm PMF}(d)$ is 
rather straightforward.  

We now focus on the calculation of $G_{\rm ele}^{\rm PMF}(d)$. Again, the double-sum term can be evaluated directly.  
Denote $\tau_0=\frac{1}{32 \pi^2 \ve_0}\left( \frac{1}{ \ve_{\rm w}} - \frac{1}{\ve_{\rm m}}\right). $ 
We have for the first two terms in  $G_{\rm ele}^{\rm PMF}(d)$  that 
\begin{align*}
&F_{\rm ele}[\phi_d]  - F_{\rm ele}[\phi_\infty] \nonumber  \\
&\quad = \tau_0 \int_\Omega f(\phi_d)  \left| \left(\sum_{i\in\text{\tiny Plate I}} + \sum_{i\in\text{\tiny Plate II}} \right) \frac{Q_i (\bx-\bx_i) }{(\bx - \bx_i)^3} \right|^2d\bx  - 2 \tau_0 \int_\Omega f(\phi_\infty) \left|   \sum_{i\in\text{\tiny Plate I}}\frac{Q_i (\bx-\bx_i) }{(\bx - \bx_i)^3} \right|^2 d\bx
\\
& \quad\quad +\tau_0 \int_{\R^3 \setminus \Omega}   \left| \left(\sum_{i\in\text{\tiny Plate I}} + \sum_{i\in\text{\tiny Plate II}} \right) \frac{Q_i (\bx-\bx_i) }{(\bx - \bx_i)^3} \right|^2d\bx  - 2 \tau_0 \int_{\R^3 \setminus \Omega}
 \left|   \sum_{i\in\text{\tiny Plate I}}\frac{Q_i (\bx-\bx_i) }{(\bx - \bx_i)^3} \right|^2 d\bx. 
 \end{align*}
 The integrals over $\Omega$ can be evaluated by numerical quadrature. Note 
that $f(\phi_d)$ and $f(\phi_\infty)$ vanish in a neighborhood of solute particles $\bx_i$ so that these integrals are well-defined.
By the symmetry and the fact that the single plate that we used for calculating $\phi_\infty$ is one of the two plates, the sum of the integrals over $\R^3 \setminus \Omega$ are simplified to
 \begin{align}
 \label{sumsum}
2\tau_0\sum_{i\in\text{\tiny Plate I}} \sum_{j\in\text{\tiny Plate II}} Q_i Q_j \int_{\mathbb{R}^3\backslash\Omega}\frac{(\bx-\bx_i)\cdot (\bx - \bx_j) }{|\bx - \bx_i|^3 |\bx - \bx_j|^3}   d\bx. 
\end{align}
For each pair $i$ and $j$ in the double-sum, we have 
\begin{align*}
&\int_{\R^3 \setminus \Omega} \frac{(\bx-\bx_i) \cdot (\bx - \bx_j) }{|\bx - \bx_i|^3 
|\bx - \bx_j|^3}  d\bx
= \int_{\R^3 \setminus \Omega} \nabla \left( \frac{1}{|\bx-\bx_i|}\right) \cdot
 \nabla \left( \frac{1}{|\bx-\bx_j|} \right)  d\bx\\
&\quad = - \int_{\partial \Omega}  \frac{1}{|\bx-\bx_i|} \, 
\frac{\partial }{\partial \bn} \left(  \frac{1}{|\bx-\bx_j|} \right)  dS_\bx
=  \int_{\partial \Omega}  
\frac{\bn(\bx) \cdot (\bx - \bx_j)} {|\bx - \bx_i| \, |\bx-\bx_j|^3} \, dS_\bx, 
\end{align*}
where $\partial/\partial \bn$ denotes the normal derivative along the boundary
$\partial \Omega$  and $\bn(\bx)$ is the unit normal to 
$\partial \Omega$ at $\bx$ pointing from inside to outside of $\Omega.$ 
By the symmetry again, we have 
\begin{align*}
2 \int_{\R^3 \setminus \Omega} \frac{(\bx-\bx_i) \cdot (\bx - \bx_j) }{|\bx - \bx_i|^3 
|\bx - \bx_j|^3}  d\bx 
= \int_{\partial \Omega} \frac{\bn(\bx)}{|\bx-\bx_i| \, |\bx - \bx_j|}  \cdot 
\left( \frac{\bx-\bx_i}{ |\bx-\bx_i|^2} + \frac{\bx-\bx_j}{ |\bx-\bx_j|^2} \right) dS_\bx. 
\end{align*}
Hence,  \reff{sumsum} is further simplified to
\begin{align*}
\tau_0 \sum_{i\in\text{\tiny Plate I}} \sum_{j\in\text{\tiny Plate II}} Q_i Q_j \int_{\partial \Omega} \frac{\bn(\bx)}{|\bx-\bx_i| \, |\bx - \bx_j|}  \cdot \left( \frac{\bx-\bx_i}{ |\bx-\bx_i|^2} + \frac{\bx-\bx_j}{ |\bx-\bx_j|^2} \right) dS_\bx,
\end{align*}
and can therefore be calculated by evaluating the surface integrals.

\vspace{4 mm}

\noindent{\bf Acknowledgments.}
Y.Z. was supported by a grant from the Simons Foundation through Grant No.\ 357963, and University Facilitating Fund from George Washington University. H.S. was supported in part by an AMS-Simons Foundation Travel Grant and Simons Foundation Collaborative Grant with grant number 522790. B.L. was supported in part by the NSF through the grant DMS-1620487. 
Q.D. was supported in part by NSF DMS-1719699.

\end{document}